\documentclass[11pt,a4paper]{article}
\usepackage[english]{babel}
\usepackage{amsfonts}
\usepackage{amsthm}

\newtheorem{teo}{Theorem}[section]
\newtheorem{prop}[teo]{Proposition}
\newtheorem{lem}[teo]{Lemma}
\newtheorem{coro}[teo]{Corollary}
\newtheorem{defi}[teo]{Definition}
\theoremstyle{definition}
\newtheorem{rem}[teo]{Remark}

\def\h{{\cal H}}
\def\b{{\cal B}}

\def\fh{ {\cal J}({\cal H})  }
\def\bh{{\cal B}({\cal H})}

\def\kh{ {\cal K}(\h) }
\def\u{ {\cal U}({\cal H}) }
\def\ufi{ {\cal U}_{\phi} }
\def\pm{ {\cal P}({\cal H})}
\def\nf{ {\|\cdot\|_{\phi}}}
\def\nfp{ {\|\cdot\|_{\phi'}}}
\def\IF{ {\cal I}_{\phi} }
\def\IFO{ {\cal I}_{\phi}^{(0)} }
\def\IFp{ {\cal I}_{\phi'} }

\def\s{\mathcal{S}}

\begin{document}

%\title[Short paths in the unitary group]{Short paths for symmetric norms in the unitary group\footnote{2010 MSC. Primary 47B10, 53C23;  Secondary 51F25, 53C22}}

\title{Short paths for symmetric norms in the unitary group}
% MSC 2010:
% 47 Operator theory. 47B10 operators belonging to operator ideals.
% 52 Topological groups, Lie groups. 22E65 infinite-dimensional Lie groups and their Lie algebras.
% 51 Geometry. 51F25 metric geometry of orthogonal and unitary groups
% 53 Differential Geometry. 53C22 geodesics
% 53 ''. 53C23 global geometric and topological methods (a la Gromov); differential geometric analysis on metric spaces.
% 58 Global analysis, analysis on manifolds. 58E10 applications to the theory of geodesics
% 
\date{}
\author{Jorge Antezana, Gabriel Larotonda and Alejandro Varela}

\maketitle%amsafterabs

\begin{abstract} For a given symmetrically normed ideal ${\cal I}_{\phi}$ on an infinite dimensional Hilbert space ${\cal H}$, we study the rectifiable distance in the classical Banach-Lie unitary group
$$
{\cal U}_{\phi}=\{u\mbox{ a unitary operator in }{\cal H}, \, u-1\in {\cal I}_{\phi}\}.
$$
We prove that one-parameter subgroups of ${\cal U}_{\phi}$ are short paths, provided the spectrum of the exponent is bounded by $\pi$, and that any two elements of ${\cal U}_{\phi}$ can be joined with a short path, thus obtaining a Hopf-Rinow theorem in this infinite dimensional setting, for a wide and relevant class of (non necessarily smooth) metrics. Then we prove that the one-parameter groups are the unique short paths joining given endpoints, provided the symmetric norm considered is strictly convex. {\bf This preprint has been replaced by \cite{alv} where the general case of convex Lagrangians in the unitary group is studied with similar techniques}.
\end{abstract}

\section{Introduction}

The group of unitary operators on Hilbert space carries, as any Banach-Lie group, a canonical connection without torsion $\nabla_XY=\frac12 [X,Y]$, whose geodesics are the one-parameter groups  $t\mapsto ue^{tz}$ (here $u$ is a unitary operator and $z$ an anti-Hermitian operator). In the finite dimensional setting, the trace is available to introduce a Riemannian metric on the unitary group in a standard fashion 
$$
\langle x,y\rangle_g=Tr(u^*x(u^*y)^*)=Tr(xy^*),
$$
for $u^*x,u^*y$ in the Lie algebra of the group, that is, for $u^*x,u^*y$ anti-Hermitian matrices. It is well-known that the connection just introduced is in fact the Levi-Civita connection of the metric $g$ induced by the trace, and that geodesics are short provided the spectrum of $z$ is bounded by $\pi$. 

Now consider the bi-invariant Finsler metric given by the uniform norm,
$$
\|x\|_u=\|u^*x\|=\|x\|
$$
for any $x$ tangent to a unitary operator $u$. Remarkably, if one keeps the connection but changes the metric, the geodesics of the connection are still short for the induced rectifiable distance (which, as in the Riemannian setting, is computed as the infimum of the length of piecewise smooth curves joining given endpoints, and $L(\gamma)=\int_0^1 \|\dot{\gamma}\|dt$). The proof of this well-known fact deserves to be sketched here since is short and elegant: if $\gamma$ is a path of unitaries such that $\gamma(0)=u,\gamma(1)=v=ue^z$, with $z$ any anti-Hermitian logarithm of $u^*v$ of norm bounded by $\pi$, let $\rho$ be a norming state for $-z^2\ge 0$ and $\xi$ be a unit norm vector in the Hilbert space ${\cal H}_{\rho}$ of the Gelfand-Neimark-Segal representation induced by $\rho$, such that $-z^2\cdot \xi =\|z\|^2\xi$ (choose $\xi$ to be the class of the identity operator). Then the map $w\mapsto w\cdot \xi$, from the unitary group to the unit sphere of ${\cal H}_{\rho}$, with its standard Riemann-Hilbert metric, sends the one-parameter  group $\delta(t)=ue^{tz}$ to a geodesic of the unit sphere, with equal lengths. Thus $ue^{tz}\cdot \xi$ is shorter than $\gamma\cdot \xi$, and 
$$
L(\gamma)=\int\|\dot{\gamma}\|\ge \int \|\dot{\gamma}\cdot\xi\|_{{\cal H}_{\rho}}\ge \int \|(ue^{tz}\cdot\xi)\dot{\,}\|_{{\cal H}_{\rho}}=\|z\|=L(\delta).
$$
This idea can be traced back to a paper by Porta and Recht \cite{pr}, where it was used to prove the minimality of certain paths in the Grassmann manifold of a $C^*$-algebra, see also \cite{odospe} for a Hopf-Rinow theorem in the context of (infinite-dimensional) Grassmann manifolds.

\medskip

In recent years, there has been an increasing interest in the geometric properties of ideals ${\cal I}$ of compact operators on Hilbert space (see \cite{belti,brgafa,dykema} and the references therein), and in particular in the Lie and representation theory of the subgroups of the linear group $G_{\cal I}$, which consist of invertible operators $g$ on ${\cal H}$ such that $g-1\in {\cal I}$. Of particular relevance is the group of unitary operators, since it acts isometrically on the ideals that arise from a symmetric norm, and gives rise to several homogeneous spaces such as the Grassmann manifold, which can be viewed as the coadjoint orbit of a fixed projection. 

\medskip

In this paper we focus on the metric properties of the unitary groups arising from a symmetric norm on compact operators. We study short curves, and we prove what we think is a fundamental result in the interplay between the algebraic and geometric properties of these groups: one-parameter subgroups are short geodesics for the rectifiable distance induced by the symmetric norm, thus extending the above mentioned result for the Finsler metric induced by the uniform norm. We also prove that these geodesics are unique among short paths joining its endpoints, provided the symmetric norm involved is strictly convex. Since the paper is intended for those interested in metric geometry and operator algebras alike, we have tried to keep the prerequisites to a minimum, and we also included the proof of some perhaps well-known facts in both areas to make the paper more readable.

\medskip

This paper is organized as follows. In Section \ref{sn} we fix our notation regarding symmetric norms and the ideals of compact operators induced. In Section \ref{distancias} we recall the notions of rectifiable length and rectifiable distance associated to a symmetric norm in $\u$. Then in Section \ref{gs} we introduce the notion of geodesic segment, which is a one-parameter group in $\ufi$ such that its initial speed has spectrum contained in $[-i\pi,i\pi]$. We prove a  certain triangle inequality valid for any symmetric norm (Proposition \ref{triad}), the key tool to the proof of Theorem \ref{distancia}. This is one of the two main results of this paper, and states that geodesics segments are short for {\em any symmetric norm}. In Section  \ref{unicas}, we recall the notion of dual norm and prove some useful properties in connection with conditional expectations. We use these tools to prove Corollary \ref{esunica}, which is the other main result of this paper, and states that if the symmetric norm is strictly convex, then geodesic segments are the {\em unique} short paths joining given endpoints in $\ufi$.

\section{Symmetric norms}\label{sn}

Let ${\cal H}$ be a complex separable and infinite dimensional Hilbert space. In this paper $\bh$, ${\cal K}({\cal H})$ and $\fh$ stand for the sets of bounded linear operators, compact operators and finite rank operators in ${\cal H}$ respectively. The unitary group of $\bh$ is indicated by $\u$, and $\pm$ is the set of self-adjoint projections in $\bh$. If $x\in \bh$, then $\|x\|$ stands for the usual uniform norm, and we will use $|\cdot|$ to indicate the modulus of an operator, i.e. $|x|=\sqrt{x^*x}$. We indicate with $\bh_{h}$ (resp. $\bh_{ah}$) the real linear spaces of Hermitian (resp. anti-Hermitian) elements of $\bh$.

\medskip

If $\nf:\bh\to \mathbb R\cup \{\infty\}$ is a norm, let
$$
\IF=\{x\in \bh: \|x\|_{\phi}<\infty\}
$$
stand for the set of operators with finite norm. A norm $\nf$ is called \textit{symmetric} if 
$$
\|xyz\|_{\phi}\le \|x\|\|y\|_{\phi}\|z\|
$$
for any $x,y,z\in bh$. The standard references on the subject are the books by Gohberg and Krein \cite{gohberg} and Simon \cite{simon}. The set $\IF$ is a bilateral ideal in $\bh$. Let 
$$
\IFO=\overline{\fh}^{\|\cdot\|_{\phi}}
$$
stand for the closure of the finite rank operators in the symmetric norm, which is again a bilateral ideal of compact operators. Clearly $\IFO\subset \IF$, while if $\IFO=\IF$ we say that $\|\cdot\|_{\phi}$ is {\em regular}. The ideals $\IFO,\IF$ are the {\em minimal} and {\em maximal} ideals related to the symmetric norm, according to the classical literature \cite[Chapter 2]{simon}. It will be assumed (unless otherwise stated) that $\|\cdot\|_{\phi}$ is {\em inequivalent} to the uniform norm of $\bh$, so $\IFO,\IF$ are proper ideals of compact operators. Some elementary useful properties include the following:
\begin{itemize}
\item {\em unitary invariance}: for any $u,v\in \u$, $\|uxv\|_{\phi}=\|x\|_{\phi}$.
\item {\em gauge invariance}: if $x=u|x|$ is the polar decomposition of $x$, then  $\|x\|_{\phi}=\|\, |x|\, \|_{\phi}$.
\item {\em invariance for the adjoint}:  for any $x$, it holds $\|x\|_{\phi}=\|x^*\|_{\phi}$.
\end{itemize}

\bigskip

What we are interested in here is in the unitary groups of $\|\cdot\|_{\phi}$, that is
$$
\ufi=\{u\in \u: u-1\in \IF\}\,\mbox{ or } \ufi^{(0)}=\{u\in \u: u-1\in \IFO\}.
$$
The topology in these groups  is given by the metric 
$$
d(u,v)=\|u-v\|_{\phi}=\|(u-1)-(v-1)\|_{\phi}.
$$
Well-known examples of symmetric norms are the $p$-norms ($p\ge 1$) given by
$$
\|x\|_p= Tr( |x|^p )^{\frac1p},
$$
where $Tr$ is the usual (infinite) trace of $\bh$, and also the Ky-Fan $k$-norms. In the case of the $p$-norms, the regular ideals induced are called {\em $p$-Schatten ideals}, and the unitary groups are examples of  what Pierre de la Harpe calls {\em classical Banach-Lie groups} of operators on Hilbert space  \cite{dharpe}. Note that for $p=\infty$, $\cal \IF=\bh$ while $\IFO=\kh$.

\section{The unitary group and its rectifiable distances}\label{distancias}

The unitary group $\u$ of $\bh$ is a smooth manifold in the uniform norm of $\bh$, and it is in fact a real Banach-Lie subgroup of the full linear group. Its Banach-Lie algebra can be identified with $\bh_{ah}$, and if $\gamma:[0,1]\to \u$ is a smooth curve, then $\gamma^*(t)  \dot{\gamma}(t)\in \bh_{ah}$ for any $t\in [0,1]$. Here smooth means $C^1$ and with non-vanishing derivative.

\medskip

It is well-known that if one endows the tangent spaces of $\u$ with the Finsler metric that consists of measuring vectors with the uniform norm of $\bh$,
$$
\|\dot{\gamma}\|_{\gamma}=\|\gamma^*\dot{\gamma}\|=\|\dot{\gamma}\|,
$$
then the curves $\delta(t)=e^{itz}$, with Hermitian symbol $z$, are shorter than any other piecewise smooth curve joining its endpoints, provided $\|z\|\le \pi$ (see the introduction for a proof). Thus the rectifiable distance related to the uniform norm is also known as {\em exponential length} in the unitary group (of a $C^*$-algebra), and in that setting it is a relevant invariant of the classification theory, see \cite{phil} and the references therein.

However, in $\bh$ (in fact, in any von Neumann algebra), any pair of elements $u,v\in \u$ are conjugated by an  Hermitian $z$ with $\|z\|\le\pi$  (i.e. there exists $z\in \bh_h$ such that $\|z\|\le \pi$ and $ue^{iz}=v$). This element is unique if $\|u-v\|<2$, and in that case $\|z\|<\pi$. Thus if one defines the rectifiable distance as the infimum of the lengths of curves joining its endpoints,
$$
d_{\infty}(u,v)=\inf\left\{ \int_0^1 \|\dot{\gamma}\| :\, \gamma\mbox{ is piecewise smooth and joins } u \mbox{ to } v \mbox{ in }\u\right\},
$$
then $d_{\infty}(u,v)=\|z\|$. 

\begin{rem}\label{compare}
Recall that there exists a constant $c_{\phi}$ such that $\|x\|\le c_{\phi}\|x\|_{\phi}$ for any symmetric norm $\nf$ and any $x\in\IF$, see for instance \cite[page 65 and the footnote at page 58]{gohberg}. The number $c_{\phi}$ is the value of $\nf$ on any one-dimensional projection. So we might as well assume from now on that $c_{\phi}=1$, i.e. $\|x\|\le \|x\|_{\phi}$ for any $x\in\IF$. In particular, convergence in $\IF$ implies convergence in the uniform norm, and a continuous curve in $\IF$ is continuous for the uniform norm.
\end{rem}

We will consider piecewise smooth curves $\gamma:[0,1]\to \ufi$. Note that $\gamma,\dot{\gamma}$ are continuous also in the uniform norm. Since $\dot{\gamma}\in \IF$, it is legitimate to define
$$
\|\dot{\gamma}\|_{\gamma}=\|\dot{\gamma}\|_{\phi}.
$$
\begin{defi}The rectifiable length of $\gamma$ is
$$
L_{\phi}(\gamma)=\int_0^1 \|\dot{\gamma}\|_{\phi},
$$
and the rectifiable distance among $u,v\in \ufi$ is
$$
d_{\phi}(u,v)=\inf\left\{ L_{\phi}(\gamma) :\, \gamma:[0,1]\to\ufi\mbox{ is piecewise smooth and joins } u \mbox{ to } v \mbox{ in }\ufi\right\}.
$$
\end{defi}

The function $d_{\phi}$ is in fact a distance, since  $\|u-v\|_{\phi}\le d_{\phi}(u,v)$ for  any $u,v\in \ufi$, see the related Proposition \ref{comple} below.

One of the main features of this metric is that is invariant for the action of the unitary group $\u$, in fact it is a bi-invariant metric
$$
d_{\phi}(uv_1w,uv_2w)=d_{\phi}(v_1,v_2)
$$
for $u,w\in \u$ and $v_1,v_2\in\ufi$.

\begin{rem}\label{logaritmo}
Observe that, for given $u\in \ufi$, $u-1$ is a normal compact operator and thus admits a spectral decomposition relative to a family of disjoint projections $\{p_k\}_{k\in\mathbb N}$, namely $u-1=\sum \alpha_k p_k$ for some bounded sequence $\alpha_k$. In fact $\lim_k\alpha_k=0$ since $u-1$ is compact. Now let $p_0=1-\sum p_k$ thus
$$
u=p_0+\sum_k (1+\alpha_k) p_k.
$$
It is then clear that $|1+\alpha_k|=1$ for any $k\in\mathbb N$. Let $\{t_k\}_{k\in\mathbb N}$ be a sequence of real numbers such that $|t_k|\le \pi$ and $e^{it_k}=1+\alpha_k$. In particular, $t_k\to_k 0$. From the elementary inequality
\begin{equation}\label{comparando}
|e^{it}-1|=\sqrt{2(1-\cos t)}\ge |t|\sqrt{1-\frac{t^2}{12}},
\end{equation}
valid for any $t\in [-\pi,\pi]$, and the fact that
$$
|t|\sqrt{1-\frac{t^2}{12}}\ge \frac12|t|
$$
for sufficiently small $t$, we derive that $|t_k|\le 2 |\alpha_k|$ for sufficiently large $k$. Thus the operator $z\in\bh$ defined by $z=\sum_k t_k p_k$ is clearly compact and moreover $z\in \IF$,  since $|\alpha_k|$ are the singular values of $u-1$ (see \cite[page 69]{gohberg}).
\end{rem}

\begin{rem}\label{loga}
The previous remark shows that for any $u\in\ufi$ there exists an Hermitian $z\in \IF$ with $\|z\|\le\pi$ such that $e^{iz}=u$. Moreover, if $\|u-1\|<2$ (or equivalently $\|z\|<\pi$), this $z$ is unique by the uniqueness of the principal branch of the analytic logarithm of $u$ in $\bh$. Then any pair of unitaries $u,v\in\ufi$ can be conjugated with an  Hermitian $z\in \IF$, $ue^{iz}=v$, with $\|z\|\le\pi$, and one can consider the elementary smooth curve
$$
\delta(t)=ue^{itz}
$$
which joins them in $\ufi$. Its rectifiable length is $\|z\|_{\phi}$. It is known, for the $p$-norms of $\bh$, with $p\ge 2$, that these curves are shorter that any other piecewise smooth curve joining $u$ to $v$, and hence
$$
d_p(u,v)=\|z\|_p,\quad p\ge 2.
$$
Moreover the curve is unique, among $C^2$ curves joining $u,v$, if $\|z\|<\pi$ and $p<\infty$. The proof of these facts can be found in \cite[Theorem 3.2]{upe}. See also \cite{upefinita} for related results on the unitary group of a finite von Neumann algebras and their homogeneous spaces.
\end{rem}

\section{Minimality of geodesic segments}\label{gs}

Let us establish some facts concerning the curves $t\mapsto ue^{itz}$ for Hermitian $z\in \IF$ joining $u$ and $v=ue^{iz}$ in $\ufi$. We will call such paths  \textit{geodesic segments}. In this section we prove that the geodesic segments are short among piecewise smooth paths in $\ufi$, when we measure them with the rectifiable distance. We begin with a technical lemma.

\begin{lem} \label{exponentes}
Let $x,y\in {\cal K}({\cal H})_h$ and assume that $e^{ix}=e^{iy}$, $\|x\|<\pi$. Then $x$ commutes with $y$, and $|x|\le |y|$. Moreover, if $\|y\|\le \pi$, then $x=y$.
\end{lem}
\begin{proof}
Let $t\in\mathbb R$. Then $e^{ity}e^{ix}e^{-ity}=e^{ity}e^{iy}e^{-ity}=e^{iy}$. This is equivalent to 
$$
\exp(e^{ity}ixe^{-ity})=e^{iy}.
$$
Differentiating at $t=0$ we obtain $\exp_{*ix}([y,x])=0$. Since $\|ix\|<\pi$, the map $\exp_{*ix}$ is invertible (see for instance \cite[Lemma 3.3]{upe}), and the first claim $[y,x]=0$ follows. Now let $\{p_k\}_{k\in\mathbb Z}\subset \pm$ be a family of disjoint projections (i.e $p_ip_j=0$ if $i\ne j$) such that $x=\sum \alpha_k p_k$, $y=\sum\beta_k p_k$, with $\alpha_k,\beta_k$ real numbers.  Thus $e^{ix-iy}=1$ implies that there exist integers $\{n_j\}_{j=1\dots m}$, and disjoint projections $\{q_j\}_{j=1\dots m}$ such that
$$
\sum_k (\alpha_k-\beta_k)p_k=2\pi\sum_j n_j q_j.
$$
If $\|y\|\le \pi$, then multiplying by a fixed $q_j$ and taking norms gives
$$
2\pi |n_j|\le \|x-y\|<2\pi,
$$
which in turn implies $n_j=0$ for any $j=1\dots m$, or equivalently $x=y$.

In the general case, we have
$$
y-x=2\pi i \sum n_j q_j.
$$
Let ${\cal A}\subset\kh$ stand for the abelian $C^*$-algebra generated by the projections $\{p_k\}_{k\in\mathbb N}$. Then $y-x\in {\cal A}$, thus $q_j\in{\cal A}$ for any $j$ because $q_j$ can be obtained via a uniform limit of polynomials in $x-y$. In particular, $q_j$ commutes with $p_k$ for any $j,k$.

From $yp_k=xp_k+2\pi \sum_j n_j q_jp_k$, squaring both sides and rearranging terms we obtain
$$
\beta_k^2 p_k=\alpha_k^2 p_k+4\pi  \sum_j n_j(\pi n_j+\alpha_k)q_j p_k.
$$
Note that $q_jp_k$ is positive and also that $n_j(\pi n_j+\alpha_k)\ge 0$ for any $j,k$, since $|\alpha_k|< \pi$. This implies
$$
\beta_k^2 p_k\ge\alpha_k^2 p_k,
$$
thus $|\beta_k|p_k\ge |\alpha_k|p_k$, and summing over $k$ yields $|y|\ge |x|$.
\end{proof}

\begin{rem}\label{maschico}
 Note that if $0\le s\le t\in \IF$, there exists $a\in \bh$ such that $\|a\|\le 1$ and $s=at$; in particular $s\in \IF$ and $\|s\|_{\phi}\le \|a\|\|t\|_{\phi} \le \|t\|_{\phi}$ for any symmetric norm $\nf$.
\end{rem}

The following proposition is our fundamental tool to prove that geodesic segments are short. Following an idea introduced by Li, Qiu and Zhang  in \cite{zhang1,zhang2} to study angular metrics in the finite dimensional Grassmannian, we prove a triangle inequality by using a remarkable result by Thompson \cite{thompson}, a result that until not so lang ago had an incomplete proof since it depends on the validity of Horn's conjecture on the spectrum of the sum of two Hermitian matrices (see \cite{fulton} and the references therein).

\begin{prop}\label{triad}
Let $x, y\in \kh_h$. If $z$ is a Hermitian operator such that $e^{i x}e^{i y}=e^{i z}$, and $\|z\|<\pi$, then there exist unitary operators $u_k$ and $v_k$ $\in\b(\h)$, for $k\in\mathbb N$, such that
$$
e^{i\, z}=\lim_{k\to\infty}e^{i\,  u_kxu_k^*+i\, v_kyv_k^*}
$$
in the uniform topology. Moreover, if $x,y\in \IF$ for some symmetric norm $\|\cdot\|_{\phi}$, then $z\in \IF$ and
$$
\|z\|_{\phi}\le \|x\|_{\phi}+\|y\|_{\phi}.
$$
\end{prop}

\begin{proof}
If $a$ and $b$ are Hermitian matrices in $M_n(\mathbb C)$ the problem was studied by Thompson \cite{thompson}. In this case he proved that there exist unitary matrices $u_n,v_n\in M_n(\mathbb C)$ such that
\begin{equation}\label{formula Thompson}
e^{i\,a}e^{i\,b}= e^{i\, u_nau_n^*+i\, v_nbv_n^*}.
\end{equation}

%We will use this finite dimensional result to
Let us approximate in the uniform norm the compact Hermitian operators $x$ and $y$ with finite range Hermitian operators $x_k,y_k\in \fh$ for $k\in\mathbb N$.
That implies that $\lim_{k\to\infty} e^{i\, x_k}=e^{i\, x}$, $\lim_{k\to\infty} e^{i\, y_k}=e^{i\, y}$ and
\begin{equation}\label{conv unif}
\lim_{k\to\infty} e^{i\, x_k} e^{i\, y_k}=e^{i\, x}e^{i\, y}=e^{i\, z}
\end{equation}
uniformly. Explicitly, since $x$ is compact, so is $|x|$, and let $\alpha_j\ge 0$ be a rearrangement in decreasing order of the eigenvalues of $|x|$ counted with multiplicity. Let $e_j$ stand for the corresponding unit norm eigenvector. Then $\{e_j\}_{j\in \mathbb N}$ is an orthonormal set, in fact an orthonormal basis of $\overline{ran|x|}$, that can be extended to an orthonormal basis $\{e_j\}\cup \{e_j'\}$ of 
$$
{\cal H}=\overline{ran|x|}\oplus ker(x).
$$ 
Then if $|x|=\sum_j \alpha_j e_j \otimes e_j$ is the singular value decomposition of $|x|$, where $e_j\otimes e_j$ stands for the one-dimensional projection $\langle e_j,\cdot\rangle e_j$, one obtains $x=\sum_j \alpha_j e_j \otimes (ue_j)$. We let  
$$
x_k=\sum\limits_{j=1}^k \alpha_j e_j \otimes (ue_j),
$$
and with a similar argument we let 
$$
y_k=\sum\limits_{j=1}^k \beta_j f_j \otimes (vf_j),
$$ 
where $y=v|y|$, $\{f_j\}$ is an orthonormal basis of $\overline{ran|y|}$, and $\beta_j\ge 0$ are the singular values of $|y|$ in decreasing order.

Then, if we denote with $R(x_k)$ the finite dimensional range of $x_k$, since $R(x_k)=Ker(x_k)^\perp$, we can identify $x_k$ as an operator from $R(x_k)$ to $R(x_k)$. The same can be done with $y_k$ and $R(y_k)$.
If we define $\s_k=R(x_k)+ R(y_k)$, then $x_k(\s_k)\subset \s_k$ and $y_k(\s_k)\subset \s_k$, and therefore $x_k, y_k\in\b(\s_k)\simeq M_n(\mathbb C)$ (where $n=dim(\s_k)$).

From the finite dimensional result (\ref{formula Thompson}), there exist $u_k, v_k$ unitary linear transformations in $\s_k$ (that means that $u_k u_k^*=p_{\s_k}$ and $v_k v_k^*=p_{\s_k}$, with $p_{\s_k}$ the orthogonal projection on $\s_k$) such that
\begin{equation}\label{formula thompson en sk}
e^{i\,x_k}e^{i\,y_k}= e^{i\, u_k x_k u_k^*+i\, v_k y_k v_k^*}.
\end{equation}

We can extend $u_k, v_k\in\b(\s_k)$ to the unitaries $\tilde u_k=u_k+p_{\s_k^\perp}$ and $\tilde v_k=v_k+p_{\s_k^\perp}$ $\in\b(\h)$. Then from the equality (\ref{formula thompson en sk}) valid in $\s_k$ we get the following in $\b(\h)$

\begin{equation}\label{formula thompson en B(H)}
e^{i\,x_k}e^{i\,y_k}= e^{i\, \tilde u_k x_k \tilde u_k^*+i\, \tilde v_k y_k \tilde v_k^*}.
\end{equation}

By hypothesis $\|z\|<\pi$, then $\|e^{i\, z}-1\|<2$, and therefore $\|e^{i\, x}e^{i\,  y}-1\|<2$. Then, taking $k_0$  large enough
$$
%\begin{equation}\label{desig con k}
\|e^{i\, x_k}e^{i\, y_k}-1\|<2, \ \ \ \hbox{ for } k\geq k_0.
$$%\end{equation}
Let $w_k$ be the unique Hermitian operator such that
$
i\, w_k=\log \left(e^{i\, x_k} e^{i\, y_k}\right)
$
and $\|w_k\|<\pi$. Therefore using this and the expression (\ref{formula thompson en B(H)})
\begin{equation}
\label{wk cumple}
e^{i\, w_k}=e^{i\, x_k} e^{i\, y_k}=e^{i\, \tilde u_k x_k \tilde u_k^*+i\, \tilde v_k y_k \tilde v_k^*}.
\end{equation}
Then using (\ref{conv unif}) and the fact that $\|z\|<\pi$ and $\|w_k\|<\pi$ we get  that
\begin{equation}
\label{wk tiende}
\lim_{k\to\infty} w_k=z,
\end{equation}
in the uniform norm (in particular $z$ is compact). Since $(\tilde u_k x\tilde u_k^*+\tilde v_k y\tilde v_k^*)- (\tilde u_k x_k\tilde u_k^*+\tilde v_k y_k\tilde v_k^*)\to 0$, then using  (\ref{wk cumple}) and (\ref{wk tiende}) we get
$$
e^{i\, (\tilde u_k x\tilde u_k^*+\tilde v_k y\tilde v_k^*)}- e^{i\, z}\to 0
$$
uniformly.

Now, for each $k\in\mathbb N$, from (\ref{wk cumple}) and Lemma \ref{exponentes} we get
$$
|w_k|\leq |\tilde u_k x_k\tilde u_k^*+\tilde v_k y_k\tilde v_k^*|.
$$
Since $x_k, y_k$ are finite range operators, then $w_k$ is also a finite range operator. From Remark \ref{maschico} above follows that
\begin{equation}\label{kfini}
\|w_k\|_{\phi}\le \|x_k\|_{\phi}+\|y_k\|_{\phi}.
\end{equation}

If $x,y\in \IF$, since $x_k=P_kx$ for some orthogonal projection $P_k\in \pm$, one obtains 
$$
\|x_k\|_{\phi}=\|P_k x\|_{\phi}\le \|P_k\| \|x\|_{\phi}=\|x\|_{\phi}.
$$
Similarly $\|y_k\|_{\phi}\le \|y\|_{\phi}$. Thus from (\ref{kfini}) follows that
$$
\|w_k\|_{\phi}\le \|x\|_{\phi}+\|y\|_{\phi},
$$
in particular
$$
\sup_k\|w_k\|_{\phi}<\infty.
$$
Since $w_k\to z$ uniformly, the noncommutative Fatou lemma \cite[Theorem 2.7(d)]{simon} gives us $z\in \IF$ and moreover
$$
\|z\|_{\phi}\le \|x\|_{\phi}+\|y\|_{\phi}.
$$
\end{proof}

\begin{coro}\label{corti}
Let $x,y,z\in {\cal I}_{\phi}$ be Hermitian, $e^{ix}e^{iy}=e^{iz}$. Then if $\|y\|<\pi$ and $\|z\|\le \pi$,
$$
\|z\|_{\phi}\le \|x\|_{\phi}+\|y\|_{\phi}.
$$
\end{coro}
\begin{proof}
Consider the triangle spanned by $1,e^{ix},e^{iz}$ in $\u$. Note that there exists $\epsilon>0$ such that
$$
\|e^{-ix}e^{isz}-1\|<2
$$
if $1-\epsilon<s\le 1$, due to the fact that $\|e^{iy}-1\|<2$. Thus $y_s\in \kh_h$ defined as $iy_s=\log(e^{-ix}e^{isz})$ is well defined in the same interval, and depends smoothly on $s$; moreover it is easy to check that $y_s\in {\cal I}_{\phi}$, and then $\|y_s-y\|_{\phi}\to 0$ when $s\to 1^-$. By the previous proposition, since $\|sz\|=s\|z\|<\pi$ for $s<1$, we obtain
$$
\|sz\|_{\phi}\le \|x\|_{\phi}+\|y_s\|_{\phi}
$$
for $s$ close to $1$. Letting $s\to 1^-$ gives the claim.
\end{proof}

\subsection{Geodesic segments are short}

Recall that we use the term \textit{segment} for the curves $ue^{itz}$ in $\ufi$, with $z\in {\cal I}_{\phi}$ Hermitian. A \textit{polygonal path} is a continuous curve that consists piecewise of segments.

\begin{prop}\label{poli}
If $z\in {\cal I}_{\phi}$ is Hermitian, $\|z\|<\pi$, and $P:[0,1]\to \ufi$ is a polygonal path joining $1$ to $e^{iz}$, then $L_{\phi}(P)\ge \|z\|_{\phi}$.
\end{prop}
\begin{proof}
Iterating the argument, it suffices to prove the assertion for a path that consists of two segments, that is, assume $e^{iz}=e^{ix}e^{iy}$ and $P$ consists of $\mu_1(t)=e^{itx}$ followed by $\mu_2(t)=e^{ix}e^{ity}$, with $\|y\|<\pi$. Now $L_{\phi}(P)\ge \|z\|_{\phi}$ is the claim of Corollary \ref{corti}.
\end{proof}

Now we approximate a smooth path with a polygonal path.

\begin{lem}
Let $\gamma:[0,1]\to\ufi$ be piecewise smooth. Then for any $\epsilon>0$ there is a polygonal path $P_{\epsilon}:[0,1]\to\ufi$ such that for any $t\in [0,1]$,
$$
\|P_{\epsilon}^*(t)\dot{P}_{\epsilon}(t)-\gamma^*(t)\dot{\gamma}(t)\|_{\phi}<\epsilon.
$$
\end{lem}
\begin{proof}
We may as well assume that $\gamma$ is smooth in $[0,1]$. Recall that $\gamma,\dot{\gamma}$ are continuous in the uniform norm. Let $\epsilon>0$, and choose a partition $0=t_0<t_1<\cdots<t_n=1$ of the interval $[0,1]$ such that, for any $k=0,1,\cdots,n$,
$$
\|\gamma(t)-\gamma(s)\|_{\phi}<2\quad\mbox{ and }\quad\|\gamma^*(t)\dot{\gamma}(t)-\gamma^*(s)\dot{\gamma}(s)\|_{\phi}<\frac{\epsilon}{2}
$$
if $s,t\in [t_k,t_{k+1}]$. Then, by Remark \ref{compare} $\|\gamma(t)-\gamma(s)\|<2$ if $s,t\in [t_k,t_{k+1}]$. Let $\log$ denote the principal branch of the analytic logarithm in $\bh$. According to Remark \ref{loga}, the unique anti-Hermitian logarithm of $\gamma^*(t_k)\gamma(t_{k+1})$ in $\bh$ such that $\|z\|<\pi$ is in fact an element of $\IF$, due to the fact that
$$
\|\gamma^*(t_k)\gamma(t_{k+1})-1\|<2.
$$
Then put
$$
z_k=\frac{\log(\gamma^*(t_k)\gamma(t_{k+1}))}{t_{k+1}-t_k}\in\IF.
$$
Now note that, for any fixed $t\in [0,1]$, the map $g:h\mapsto \frac1h \log(\gamma^*(t)\gamma(t+h))$, is well-defined and analytic, for sufficiently small $h$. Moreover, if $h\to 0$, then
$$
g(h)\to d\log_1 \gamma^*(t)\dot{\gamma}(t)=\gamma^*(t)\dot{\gamma}(t).
$$
Then, taking a refinement of the partition if necessary, we can also assume that
$$
\|z_k-\gamma^*(t_k)\dot{\gamma}(t_k)\|_{\phi}<\frac{\epsilon}{2}
$$
for any $k=0,1,2\cdots, n$. Consider the map $P_{\epsilon}:[0,1]\to\u$ which is defined as
$$
P_{\epsilon}(t)=\gamma(t_k)e^{(t-t_k)z_k}\mbox{ for }t\in [t_k,t_{k+1}].
$$
Then $P_{\epsilon}$ is certainly a polygonal path, and it is straightforward to see that verifies the claim of the lemma.
\end{proof}

\begin{teo}\label{distancia}
Let $u,v\in \ufi$ and $v=ue^{iz}$, with $\|z\|\le \pi$, $z\in \IF$ and Hermitian. Then, the curve $\delta(t)=ue^{itz}$ is shorter than any other piecewise smooth curve $\gamma$ in $\ufi$ joining $u$ to $v$, when we measure them with the norm $\nf$. In particular, $d_{\phi}(u,v)=\|z\|_{\phi}$.
\end{teo}
\begin{proof}
It suffices to prove the theorem for $u=1$. Assume first that $\|z\|<\pi$. For given $\epsilon>0$, let $P_{\epsilon}$ be a polygonal path in $\ufi$ as in the previous lemma, joining $1$ to $e^{iz}$, such that
$$
\|\gamma^*\dot{\gamma}-P_{\epsilon}^*\dot{P}_{\epsilon}\|_{\phi}<\epsilon.
$$
Let $\delta(t)=e^{itz}$. Then by Proposition \ref{poli},
$$
\|z\|_{\phi}=L_{\phi}(\delta)\le L_{\phi}(P_{\epsilon})=\int_0^1 \|P_{\epsilon}^*\dot{P}_{\epsilon}\|_{\phi}\le \int_0^1 \|\gamma^*\dot{\gamma}-P_{\epsilon}^*\dot{P}_{\epsilon}\|_{\phi}+ L_{\phi}(\gamma)<\epsilon+L_{\phi}(\gamma),
$$
and thus $L_{\phi}(\gamma)\ge \|z\|_{\phi}$. 

If $\|z\|=\pi$, consider for $s\in(0,1)$ the path $\delta_s$ which is the geodesic segment joining $\gamma(s)$ to $e^{isz}$. Then
$$
\int_0^1 \|\dot{\gamma}\|_{\phi}=\int_0^s \|\dot{\gamma}\|_{\phi}+\int_s^1 \|\dot{\gamma}\|_{\phi}=\int_0^s \|\dot{\gamma}\|_{\phi}+L_{\phi}(\delta_s)+\int_s^1 \|\dot{\gamma}\|_{\phi}-L_{\phi}(\delta_s).
$$
Note that for $s<1$, the path $\gamma|_{[0,s]}$ followed by $\delta_s$ is a piecewise smooth curve joining $1$ to $e^{isz}$. Since $\|sz\|=s\|z\|<\pi$, by what we have just proved,
$$
\int_0^s \|\dot{\gamma}\|_{\phi}+L_{\phi}(\delta_s)\ge s\|z\|_{\phi}.
$$
Also note that, for $s$ close to $1$, $\delta_s(t)=\gamma(s)e^{itz_s}$, with $iz_s=\log(\gamma^*(s)e^{isz})$. Since $\|\gamma(s)-e^{iz}\|_{\phi}\to 0$ as $s$ approaches $1$ from the left, then $\|z_s\|_{\phi}\to 0$  as $s\to 1^-$. Then
$$
L_{\phi}(\delta_s)=\int_0^1 \|\dot{\delta_s}\|_{\phi}=\int_0^1 \|z_s\|_{\phi}
$$
shows that $L_{\phi}(\delta_s)\to 0$ as $s\to 1^-$. Now from
$$
L_{\phi}(\gamma)=\int_0^1 \|\dot{\gamma}\|_{\phi}\ge s\|z\|_{\phi}+\int_s^1 \|\dot{\gamma}\|_{\phi}-L_{\phi}(\delta_s)
$$
we obtain $L_{\phi}(\gamma)\ge \|z\|_{\phi}$ letting $s\to 1^-$.
\end{proof}

\subsection{Relation with the linear metric}

Here we establish that the rectifiable distance is in fact uniformly equivalent to the linear metric in the unitary group $\ufi$.

\begin{prop}\label{comple}
Let $u,v\in \ufi$. Then
$$
\sqrt{1-\frac{\pi^2}{12}}\, d_{\phi}(u,v)\le \|u-v\|_{\phi}\le d_{\phi}(u,v).
$$
In particular, the function $d_{\phi}$ is a distance in $\ufi$, equivalent to the linear distance, and if the ideal $\IF$ is complete, then the metric space $(\ufi,d_{\phi})$ is complete.
\end{prop}
\begin{proof}
By the invariance of the metric, it suffices to check the inequalities of the assertion for $v=1$. From equation (\ref{comparando}) in Remark \ref{logaritmo}, it clearly follows that if $u=e^{iz}$, with $\|z\|\le \pi$ and $z\in \IF$ Hermitian, then the inequality on the left holds. The other inequality is trivial. 

As for the completeness, it suffices to check that $\ufi$ is complete with the linear metric induced by the norm $\nf$. This follows from the fact that if $\{u_n\}\subset \ufi$ is a Cauchy sequence, then the sequence $\{u_n-1\}$ is a Cauchy sequence in $\IF$, and by assumption it converges to an element $z\in\IF$. Let $u=z+1$, we claim that $u\in \u$. Note that $\|u_n-u\|=\|(u_n-1)-z\| \le \|(u_n-1)-z\|_{\phi}\to 0$, hence

$$
\|uu^*-1\|= \lim_n \|u_nu_n^*-1\|=0,
$$
thus $uu^*=1$. Likewise, $u^*u=1$.
\end{proof}

\section{Uniqueness of short geodesics}\label{unicas}

In this section we prove uniqueness of short paths, assuming the symmetric norm is strictly convex.

\subsection{Dual norms and dual ideals}

\begin{defi}
Let $Tr$ stand for the infinite trace of $\bh$. Let $\nfp$ be the dual norm of $\nf$, which is computed for $x\in \kh$ as
$$
\|x\|_{\phi'}=\sup \{ |Tr(xy)| : \|y\|_{\phi}\le 1\}.
$$
This formula defines a norm $\nfp:\kh\to \mathbb R\cup \{\infty\}$, all the properties of a norm are trivially verified with the exception that $\|x\|_{\phi'}=0$ implies $x=0$. But this follows from the singular values expansion
$$
x=u|x|=\sum_j \alpha_j e_j\otimes (ue_j).
$$
By picking $y=(e_n\otimes e_n)u^*$ (see Remark \ref{compare}) one obtains $\alpha_n=0$ from $\|x\|_{\phi'}=0$, thus $x=0$. 
\end{defi} 

\begin{rem}\label{directly}Directly from the definition it follows that
$$
|Tr(xy)|\le \|x\|_{\phi}\|y\|_{\phi'}.
$$
Let $\|\cdot\|_{\phi''}=\|\cdot\|_{(\phi')'}$, and ${\cal I}_{\phi''}$ the corresponding ideal. Then it also follows directly from the definitions that
$$
\|x\|_{\phi''}\le \|x\|_{\phi}.
$$
In particular, if $x\in \IF$ then $x\in {\cal I}_{\phi''}$. Note that, by definition, $\IFp,{\cal I}_{\phi''}\subset\kh$.
\end{rem} 

\begin{rem}
Let $x\in \kh$, $x\ge 0$. Let $E:\kh\to C^*(x)$ stand for the unique conditional expectation compatible with the trace of $\bh$, which is obtained as follows: let $\{e_k\}$ be a basis of $ran(x)$, completed to a basis of $\cal H$. Let $p_k=e_k\otimes e_k$ be the induced family of projections. Then for any $y\in \kh$, 
$$
Tr(p_ky)=\sum_n \langle p_k ye_n,e_n\rangle=\langle ye_k,e_k\rangle
$$
thus $Tr(p_ky)\to 0$ and
$$
E(y)=\sum_k Tr(p_ky)p_k=\sum_k p_kyp_k,
$$
which is known as a {\em pinching operator}. It is well-known that pinchings in finite dimensions are contractive for the symmetric norms, the result extends easily to our context, a property that turns out to be quite useful to estimate the norm of an element.
\end{rem}

The following lemma collects a series of useful facts regarding symmetric norms and their duals.

\begin{lem}Let $\nf$ by a symmetric norm in $\bh$, let $\nfp$ stand for its dual norm. Then
\begin{enumerate}
\item $\nfp$ is a symmetric norm in $\kh$. 
\item Let $x\in \kh$ be positive, let $E$ stand for the induced conditional expectation. Then for any symmetric norm $\nf$,  we have that $E(y)\in \IF$ whenever $y\in  \IF$, and moreover
$$
\|E(y)\|_{\phi}\le \|y\|_{\phi}.
$$
\item Let $x,y\in \bh$, then
$$
|Tr(xy)|\le \|xy\|_1\le \|x\|_{\phi}\|y\|_{\phi'}.
$$
If $x,y\in \bh_h$ (or $x,y\in\bh_{ah}$), then $Tr(xy^*)\in\mathbb R$. Thus,
$$
-\|x\|_{\phi}\|y\|_{\phi'}\le Tr(xy^*)\le \|x\|_{\phi}\|y\|_{\phi'}.
$$
\item If $x\ge 0$, then $\|x\|_{\phi'}=\sup\{ Tr(xy):\,  y\ge 0,\, yx=xy,\, \|y\|_{\phi}\le 1\}$. 
\item $\IF={\cal I}_{\phi''}$ and $\|\cdot\|_{\phi''}=\nf$.
\item If $x\in \IF$, then $\|x\|_{\phi}=\|\,|x|\, \|_{\phi}=\sup\{ Tr(|x|y):\,  y\ge 0,\, y|x|=|x|y,\, \|y\|_{\phi'}\le 1\}$.
\end{enumerate}
\end{lem}
\begin{proof}
\begin{enumerate}
\item  Let $x\in \kh$, $u\in\u$. Then, since $y\mapsto yu$ is an isometric isomorphism for the norm $\nf$, naming $z=yu$ we obtain
$$
\|ux\|_{\phi'}=\sup \{ |Tr(uxy)| : \|y\|_{\phi}\le 1\}=\sup \{ |Tr(xz)| : \|z\|_{\phi}\le 1\}=\|x\|_{\phi'},
$$
that is $\|ux\|_{\phi'}=\|x\|_{\phi'}$, likewise, $\|xv\|_{\phi'}=\|x\|_{\phi'}$ for any $v\in \u$. Thus $\nfp$ is an unitarily invariant norm. Now assume that $\|x\|<1$, then \cite{kadimean} there exists $u_k\in \u$, $k=1\ldots n$ such that $x=\frac1k \sum_{k=1}^n u_k$. Thus
$$
\|xy\|_{\phi'}\le\frac1k \sum_{k=1}^n \|u_k y\|_{\phi}=\|y\|_{\phi'},
$$
hence $\|\frac{tz}{\|z\|}y\|_{\phi'}\le \|y\|_{\phi'}$ for any $t\in (0,1)$, namely $t\|zy\|_{\phi'}\le \|z\|\|y\|_{\phi'}$. Letting $t\to 1^+$ gives $\|zy\|_{\phi'}\le \|z\|\|y\|_{\phi'}$, and with a similar argument one obtains the corresponding property for the right multiplication.
\item Let $x_n=\sum_{j=1}^n \alpha_j e_j\otimes e_j$ be the finite rank operators associated to the singular value decomposition of $x$, $\|x_n-x\|\to 0$ as $n\to \infty$. Let $p_j=e_j\otimes e_j\in \pm$ and $p=\sum_{j=1}^n p_j\in \pm$. Note that ${\cal M}_p=p\bh p$ which is a finite dimensional algebra. Let $E_n:{\cal M}_p\to {\cal M}_p$ stand for the pinching
$$
E_n(pyp)=\sum_{j=1}^n p_jyp.
$$
Then, since any pinching is majorized by the identity (see for instance \cite[page 97]{bhatia}), by considering the restriction of $\nf$ to ${\cal M}_p$, we obtain
$$
\|E_n(pyp)\|_{\phi}\le \|pyp\|_{\phi}\le \|y\|_{\phi}.
$$
Note that, for any $y\in \bh$, $E_n(pyp)=E_n(y)$, thus if $y\in \IF$,
$$
\sup_n\|E_n(y)\|_{\phi}\le \|y\|_{\phi}.
$$
Since $E_n(y)\to E(y)$ strongly, by the noncommutative Fatou Lemma \cite[Theorem 2.7(d)]{simon}, we obtain that $E(y)\in \IF$ and $\|E(y)\|_{\phi}\le \|y\|_{\phi}$.
\item If $x,y\in \bh$, it is well-known that $|Tr(xy)|\le Tr|xy|=\|xy\|_1$, and then
$$
\|xy\|_1=\sup\{|Tr(uxy)|:u\in {\u}\}\le \|ux\|_{\phi}\|y\|_{\phi'}=\|x\|_{\phi}\|y\|_{\phi'}.
$$
\item If $x\ge 0$, consider the conditional expectation $E$ as in the second item. Then $Tr(xy)=Tr(E(xy))=Tr(xE(y))$ since $E$ is a bi-module map  compatible with the trace. Since $E$ is contractive,
$$
\{E(y):\|y\|_{\phi'}\le 1\}\subset \{z\in C^*(x): \|z\|_{\phi'}\le 1\},
$$
hence
\begin{eqnarray}
\|x\|_{\phi'}&=&\sup\limits_{\|y\|_{\phi}\le 1} |Tr(xy)|=\sup\limits_{\|y\|_{\phi}\le 1} |Tr(xE(y))|\le \sup\limits_{\stackrel{z\in C^*(x)}{\|z\|_{\phi}\le 1}} |Tr(xz)|\nonumber\\
&\le & \sup\limits_{ \stackrel{z\in C^*(x)}{\|z\|_{\phi}\le 1} } \|zx\|_1 = \sup\limits_{ \stackrel{z\in C^*(x)}{\|z\|_{\phi}\le 1} } Tr(x|z|), \nonumber
\end{eqnarray}
where we have used that $x$ and $|z|$ commute and thus $x|z|=|xz|\ge 0$. Hence
$$
\|x\|_{\phi'}\le\sup\{ Tr(xy):\,  y\ge 0,\, yx=xy,\, \|y\|_{\phi}\le 1\},
$$
and the other inequality is obvious.
\item Let $x\in \IF$, then as we mentioned $x\in {\cal I}_{\phi''}$ and $\|x\|_{\phi''}\le \|x\|_{\phi}$ follow from the definition of dual norms. Now assume that  $x\in {\cal I}_{\phi''}$, let $x_n,p,{\cal M}_p$ be as in the proof of the second item. Let $\IF^p=({\cal M}_p,\nf)$ which is a finite dimensional algebra. By the Hahn-Banach theorem, there exists $\varphi\in (\IF^p)^{\ast}$ (the dual space) such that $\|\varphi\|=1$ and $\varphi(x_n)=\|x_n\|_{\phi}$. Being a finite dimensional algebra, there exists $z\in {\cal M}_p$ such that $\varphi(\cdot)=Tr(z\cdot)$. It is easy to check (using the definition of dual norms) that
$$
\|z\|_{\phi'}=\|\varphi\|=1.
$$
Thus
$$
\|x_n\|_{\phi}=\varphi(x_n)=Tr(x_nz)\le \|x_n\|_{\phi''}=\|pxp\|_{\phi''}
$$
by definition of the second dual norm. Hence
$$
\|x_n\|_{\phi}\le \|pxp\|_{\phi''}\le \|x\|_{\phi''}<\infty.
$$
By the noncommutative Fatou lemma, since $x_n\to x$ in the strong operator topology, $x\in {\cal I}_{\phi}$ and
$$
\|x\|_{\phi}\le \sup_n\|x_n\|_{\varphi}\le \|x\|_{\phi''}.
$$
\item This assertion follows directly from the two previous items.
\end{enumerate}
\end{proof}

Let $\IF^{\ast}$ stand for the dual space of the normed ideal $\IF$.

\begin{coro}\label{normi}
Let $x^*=x\in \IF$. Then there exist a net $\{\xi_i\}\subset  \IFp$ such that $\|\xi_i\|_{\phi'}\le 1$, $\xi_i^*=\xi_i$, $\xi_i x=x\xi_i$, and a functional $\eta_x\in \IF^{\ast}$ such that $\eta_i=Tr(\xi_i \cdot)\to \eta_x$ in the $\omega^*$-topology, and $\eta_x(x)=\|x\|_{\phi}$.
\end{coro}
\begin{proof}
For each $y\in \IFp$, let $\eta_y(z)=Tr(yz)$, then $\eta_y\in \IF^{\ast}$ and $\|\eta_y\|\le \|y\|_{\phi'}$. Consider the family
$$
{\cal F}_x=\{\eta_y: y\in \IFp,\,y\ge 0,\, y|x|=|x|y,\, \|y\|_{\phi'}\le 1\}.
$$
By the last item of the previous lemma, there exists a sequence  $\{y_j\}\subset {\cal F}_x$ such that
$$
\|x\|_{\phi}=\lim\limits_j Tr(|x|y_j)
$$
Now if $x^*=x$, write down $x=|x|u=u|x|$ (polar decomposition) and put
$$
\xi_j=\frac12 (uy_j+y_ju)\in \IFp.
$$
Note that $\|\xi_j\|_{\phi'}\le 1$, $\xi_j^*=\xi_j$ since $u^*=u$ and also that $\xi_j$ commutes with $x$. Moreover
$$
Tr(x\xi_j)=\frac12[ Tr(uy_j|z|u)+Tr(y_j u u|z|)]=Tr(|x|y_j),
$$
thus $\|x\|_{\phi}=\lim_j Tr(x\xi_j)$. Since $\eta_j\in \IF^{\ast}$ and $\|\eta_j\|\le 1$, by the Banach-Alaouglu-Bourbaki theorem, the family $\{\eta_j\}$ has a $\omega^*$-convergent subnet $\eta_i$ to a functional $\eta_x\in \IF^{\ast}$,
$$
\eta_x(z)=\lim\limits_i \eta_i(z)=\lim_i Tr(\xi_i z)
$$
for any $z\in \IF$. In particular  $\eta_x(x)=\|x\|_{\phi}$.
\end{proof}

\subsection{Uniqueness of short paths for strictly convex norms}

\begin{defi}
Let $X$ be a Banach space. A norm  $\|\cdot\|$ in $X$ is
\begin{enumerate}
 \item \textit{strictly convex} (or \textit{rotund}) if $\|x+y\|=\|x\|+\|y\|$ implies $x=\lambda y$.
\item \textit{smooth} if it is Gateaux differentiable (at each $x\ne 0\in X$) as a map from $X$ to $\mathbb R$. Equivalently, for any $x\in X$ there exists a unique $\varphi_x\in X^*$ such that $\|\varphi_x\|=1$ and $\varphi_x(x)=\|x\|$.
%\item \textit{Frechet} differentiable if it is smooth and the map $J:x\mapsto \varphi_x$, $J:X-\{0\}\to X^*$ is norm continuous.
\end{enumerate}
\end{defi}

It is well-known that if the norm on the dual space $X^*$ is rotund, the norm in $X$ is smooth, and that if the norm in the dual space $X^*$ is smooth, then the norm in $X$ is rotund. The reciprocal of these affirmations are false, see for instance \cite[2.36]{phelps} and the references therein.

\begin{teo}
Assume that the norm $\nf$ is rotund. Let $u,v,w\in\ufi$ such that
$$
d_{\phi}(u,v)=d_{\phi}(u,w)+d_{\phi}(w,v),
$$
and assume that $d_{\infty}(u,v)<\pi$ (equivalently, $\|u-v\|<2$). Then $u,v,w$ are aligned in $\ufi$. That is, there exists $z^*=z\in \IF$ with $\|z\|<\pi$ such that
$$
v=ue^{iz},\qquad \mbox{ while }w=ue^{it_0z}
$$
for some $t_0\in [0,1]$.
\end{teo}
\begin{proof}
We assume that $u=1$. Write $v=e^{iz}$ with $\|z\|<\pi$. Let $w=e^{ix}
=e^{iz}e^{-iy}$, with $\|y\|\le \pi$, $\|x\|\le \pi$ and $x,y$ Hermitian operators in $\IF$. That is $v=e^{iz}=e^{ix}e^{iy}$. Then our assumption is
$$
\|z\|_{\phi}=\|x\|_{\phi}+\|y\|_{\phi}.
$$
Note that $\|y\|_{\phi}\le \|z\|_{\phi}$, likewise, $\|x\|_{\phi}\le \|z\|_{\phi}$. To prove the claim of the theorem, it suffices to show that $y=\lambda z$, since in that case,
$$
w=e^{ix}=e^{iz}e^{-iy}=e^{i(1-\lambda)z},
$$
and we can pick $t_0=1-\lambda$. In fact, if $y=\lambda z$, then since $\|y\|_{\phi}\le \|z\|_{\phi}$, we obtain $|\lambda|\le 1$, thus $t_0=1-\lambda\ge 0$, and on the other hand, since
$$
(1-\lambda)\|z\|_{\phi}=\|x\|_{\phi}\le \|z\|_{\phi}
$$
we also have $t_0=1-\lambda\le 1$.

We will use a standard variation technique to show that $y=\lambda z$. Consider the geodesic triangle in $\ufi$ spanned by $1,e^{ix}$ and $e^{iz}=e^{ix}e^{iy}$. Let
$$
\beta(s)=e^{ix}e^{i(1-s)y}=e^{iz}e^{-isy}.
$$
Since $\|z\|<\pi$, if $s\ge 0$ is small enough then
$$
\|\beta(s)-1\|\le \|e^{iz}-1\|+\|e^{isy}-1\|<2.
$$
Thus for small $s\ge 0$, the element $z_s=(-i)\log(\beta(s))$ is an analytic function of $s$, is Hermitian in $\IF$ and $\|z_s\|<\pi$. In particular, $z_0=z$, and moreover from $e^{i z_s}=e^{iz}e^{-isy}$ we obtain
$$
\frac{d}{ds}\vrule\mathop{}_{\;s=0}e^{iz_s}=-e^{iz} iy.
$$
If we recall the formula for the differential of the exponential map in the unitary group
$$
d\exp_x(y)=\int_0^1 e^{(1-t)x} ye^{tx}dt
$$
for $x,y\in \bh_{ah}$ (see for instance \cite[Lemma 3.3]{upe}), we obtain
\begin{equation}\label{laderi}
-y=\int_0^1 e^{-tz}\dot{z}_0 e^{tz}dt.
\end{equation}
Now by Theorem \ref{distancia}, $\|z_s\|_{\phi}=d_{\phi}(1,\beta(s))$, and by the triangle inequality for the distance $d_{\phi}$, we have
$$
\|x\|_{\phi}+(1-s)\|y\|_{\phi}=\|z\|_{\phi}-s\|y\|_{\phi}\le \|z_s\|_{\phi}\le \|x\|_{\phi}+(1-s)\|y\|_{\phi}.
$$
Thus, for any $s\ge 0$ sufficiently small,
$$
\frac{-\|z_s\|_{\phi}+\|z\|_{\phi}}{s}=\|y\|_{\phi}.
$$
Since $z_s=z+\dot{z}_0s+o(s^2)$, then $\|z_s\|_{\phi}\ge \|z+\dot{z}_0s\|-o(s^2)$ for small $s\ge 0$, and it follows that
$$
\|y\|_{\phi}\le \frac{\|z\|_{\phi}-\|z+\dot{z}_0s\|}{s}+o(s).
$$
Let $\{\xi_i\}\subset\IFp$ be a net as in Corollary \ref{normi} such that $\|z\|_{\phi}=\lim_i Tr(z\xi_i)=\eta_z(z)$. Then
$$
\|z+\dot{z}_0s\|\ge \tau (\xi_i (z+\dot{z}_0s))=\tau(\xi_i z)+s\tau(\dot{z}_0\xi_i).
$$
Then
$$
\|y\|_{\phi}\le  \frac{\|z\|_{\phi}-\tau( \xi_i z)}{s}- \tau(\xi_i \dot{z}_0)+ o(s).
$$
Since the $\xi_i$ and $z$ commute, by equation (\ref{laderi}) we have $\tau(\xi_i y)=-\tau(\xi_i \dot{z}_0)$ for any $i$. Taking limit on $i$ we obtain
$$
\|y\|_{\phi}\le   \eta_z(y)+ o(s)\le \|y\|_{\phi}+o(s)
$$
since $\|\eta_z\|\le 1$. Letting $s\to 0^+$ we obtain $\|y\|_{\phi}=\eta_z(y)$. To finish, note that
$$
\|z+y\|_{\phi}\ge \eta_z(z+y)=\eta_z(z)+\eta_z(y)=\|z\|_{\phi}+\|y\|_{\phi},
$$
which shows that $y=\lambda z$, since we are assuming that the norm $\nf$ is rotund.
\end{proof}

\begin{coro}\label{esunica}
If the norm $\nf$ is rotund, and $z$ is a Hermitian element of $\IF$ such that $\|z\|<\pi$, then $\delta(t)=ue^{itz}$ is the unique piecewise $C^1$ curve in $\ufi$ joining $u$ to $v=ue^{iz}$ with  length equal to $d_{\phi}(u,v)=\|z\|_{\phi}$. 
\end{coro}
\begin{proof}
Assume that $\gamma$ is any short, piecewise smooth curve joining $1$ to $e^{iz}$. Let $t_0\in (0,1)$ and let $\gamma(t_0)=e^{ix}=e^{iz}e^{-iy}$, with $\|y\|\le \pi$, $\|x\|\le \pi$. By Proposition \ref{poli} and Theorem \ref{distancia} applied to each segment,
$$
\|z\|_{\phi}\le \|x\|_{\phi}+\|y\|_{\phi}\le \int_0^{t_0}\|\dot{\gamma}\|_{\phi} +\int_{t_0}^1\|\dot{\gamma}\|_{\phi}=L_{\phi}(\gamma)=\|z\|_{\phi},
$$
hence $\|x\|_{\phi}+\|y\|_{\phi}=\|z\|_{\phi}$. Now the claim of the corollary follows form the previous theorem.
\end{proof}

\begin{rem}
This result generalizes  \cite[Theorem 3.2]{upe} for the $p$-norms ($p\ge 2$), where it was proved using a particular variational calculus for the $p$-norms, see also \cite{mlr}.
\end{rem}

\begin{rem}
If $\|u-v\|=2$ (equivalently, if $v=ue^{iz}$ with $\|z\|=\pi$), then the segment $\delta(t)=ue^{itz}$ is not unique as a short path joining $u,v$. In fact, it suffices to consider the case $v=1$. Since the condition $\|u-1\|=2$ is equivalent to $-2\in \sigma(u-1)$, if $x=u-1\in \IF\subset\kh$, this is equivalent to $-1\in\sigma(x)$. Then, as in Remark \ref{logaritmo},
$$
u=1+x=p_0+\sum_{k\ge 1} e^{i t_k}p_k=1+\sum_{k\ge 1}(e^{it_k}-1)p_k
$$
where $\{p_k\}_{k\ge 0}$ are disjoint projections (finite dimensional for $k\ge 1$, and $|t_k|\le\pi$, $t_k\to 0$). Let $p_1$ be the finite rank projection to the eigenspace of $\lambda=-1$, then we can assume that $t_1=\pi$, and $|t_k|<\pi$ for $k\ge 2$.  If the rank of $p_0$ is one, there are exactly two logarithms of $u$ in $\IF$ (with uniform norm equal to $\pi$), corresponding to 
$$
z={\mp} i \pi p_1+i\sum_{k\ge 2} t_k p_k.
$$
However, if $rank(p_0)=n\ge 2$, then there are an infinite number of decompositions $p_1=q+q'
$ into proper disjoint projections, which allows to choose an infinite number of logarithms $z\in \IF$,
$$
z=  i \pi q -i\pi q'+  i\sum_{k\ge 1} t_k p_k
$$
of norm equal to $\pi$.
\end{rem}

\begin{rem}
Clearly, these results of existence and uniqueness of short geodesics joining given endpoints (Theorem \ref{distancia} and Corollary \ref{esunica}) hold verbatim if we replace $\ufi$ by $\ufi^{(0)}$, the group associated to the minimal ideal which is the closure in the symmetric norm of finite rank operators. The same can be said for symmetric norms in matrix algebras, that is, if the dimension of the Hilbert space $\cal H$ is finite, a result which has interest on its own right.
\end{rem}

\bigskip

\vspace*{1cm}

\noindent

\hspace*{.4cm}\begin{minipage}{8cm}
Jorge Antezana:\\
Universitat Aut\'onoma de Barcelona.\\
Departamento de Matem\'atica,\\
Facultad de Ciencias\\
Edificio C Bellaterra (08193) \\
Barcelona, Espa\~{n}a.\\
e-mail: jaantezana@mat.uab.cat\\

\bigskip

Gabriel Larotonda and Alejandro Varela:\\
Instituto de Ciencias \\
Universidad Nacional de General Sarmiento. \\
J. M. Gutiérrez 1150 \\
(B1613GSX) Los Polvorines, \\
Buenos Aires, Argentina.  \\
e-mails: glaroton@ungs.edu.ar,\\
avarela@ungs.edu.ar\\

\end{minipage}
\hspace*{.1cm}
 \begin{minipage}{6cm}
 \vspace*{-4.7cm}
 J. Antezana, G. Larotonda\\
 and A. Varela:\\
 Instituto Argentino de Matem\'atica\\
 ``Alberto Calder\'on'', CONICET\\
 Saavedra 15, 3er piso\\
 (C1083ACA) Buenos Aires,\\
 Argentina.\\
 
 \end{minipage}


\begin{thebibliography}{XX}

\bibitem{upefinita} E. Andruchow, E. Chiumiento, G. Larotonda, {\it Homogeneous manifolds from noncommutative measure spaces}. J. Math. Anal. Appl. (2010), in press.

\bibitem{odospe} E. Andruchow, G. Larotonda, {\it Hopf-Rinow theorem in the Sato Grassmannian}, J. Funct. Anal. 255 (2008), no.7, 1692-1712.


\bibitem{unifred} E. Andruchow, G. Larotonda, {\it The rectifiable distance in the unitary Fredholm group}. Studia Math. 196 (2010), 151-178.

\bibitem{upe}  E. Andruchow, G. Larotonda, L. Recht, {\it Finsler geometry and actions of the Schatten unitary groups}, Trans. Amer. Math. Soc. 62 (2010), 319-344.

\bibitem{alv}  J. Antezana, G. Larotonda, A. Varela, {\it Optimal paths for symmetric actions in the unitary group}, preprint  arXiv:1107.2439v1 (2011).

\bibitem{belti} D. Beltita, {\it Iwasawa decompositions of some infinite-dimensional Lie groups}. Transactions of the American Mathematical Society 361 (2009), no. 12, 6613-6644.

\bibitem{brgafa} D. Beltita, T.S. Ratiu, {\it Symplectic leaves in real Banach Lie-Poisson spaces}. Geometric and Functional Analysis 15 (2005), no. 4, 753-779.

\bibitem{bhatia} R. Bhatia. Matrix analysis. Graduate Texts in Mathematics, 169. Springer-Verlag, New York, 1997.

\bibitem{dykema} K. Dykema, T. Figiel, G. Weiss, M. Wodzicki, {\it Commutator structure of operator ideals}, Adv. in Math. 185 (1) (2004), 1--79.

\bibitem{fulton} W. Fulton, {\it Eigenvalues, invariant factors, highest weights, and Schubert calculus}. Bull. Amer. Math. Soc. (N.S.)  37  (2000), no. 3, 209--249 (electronic).

\bibitem{gohberg} I. C. Gohberg, M. G. Krein.  Introduction to the theory of linear nonselfadjoint operators. Translated from the Russian by A. Feinstein. Translations of Mathematical Monographs, Vol. 18 American Mathematical Society, Providence, R.I. 1969.

\bibitem{kadimean} R.V. Kadison, G.K. Pedersen, {\it Means and convex combinations of unitary operators}. Math. Scand. 57 (1985), no. 2, 249--266.

\bibitem{dharpe} P. de la Harpe. Classical Banach-Lie algebras and Banach-Lie groups of operators in Hilbert space. Lecture Notes in Mathematics, Vol. 285. Springer-Verlag, Berlin-New York, 1972.

\bibitem{mlr} L.E. Mata-Lorenzo, L. Recht, {\it Convexity properties of ${\rm Tr}[(a\sp *a)\sp n]$}. Linear Algebra Appl. 315 (2000), no. 1-3, 25--38.

\bibitem{phelps} R. Phelps. Convex functions, monotone operators and differentiability. Second edition. Lecture Notes in Mathematics, 1364. Springer-Verlag, Berlin, 1993.

\bibitem{phil} N.C. Phillips. A survey of exponential rank, C* -Algebras: 1943-1993, A Fifty Year Celebration. Contemporary Math., Vol. 167, Amer. Math. Soc, Providence RI, 1994,
353-399.

\bibitem{pr} H. Porta, L. Recht, {\it Minimality of geodesics in Grassmann manifolds}, Proc. Amer. Math. Soc. 100 (1987), no. 3, 464-466.

\bibitem{simon} B. Simon. Trace ideals and their applications. Second edition. Mathematical Surveys and Monographs, 120. American Mathematical Society, Providence, RI, 2005.

\bibitem{thompson} R.C. Thompson, {\it Proof of a conjectured exponential formula}. Linear and Multilinear Algebra  19  (1986), no. 2, 187-197.

\bibitem{zhang1} C.-K. Li, L. Qiu, Y. Zhang,  {\it Unitarily invariant metrics on the Grassmann space}. SIAM J. Matrix Anal. Appl.  27  (2005), no. 2, 507--531 (electronic).

\bibitem{zhang2} L. Qiu, Y. Zhang, {\it On the angular metrics between linear subspaces}.  Linear Algebra Appl. 421 (2007), no. 1, 163--170.

\end{thebibliography}
\end{document}